\documentclass[10pt]{amsart}
\usepackage{amsmath,amssymb,amsfonts, amscd,hyperref}
\usepackage{graphicx}
\usepackage{amsbsy}
\usepackage{latexsym}
\usepackage{amsthm}
\usepackage{amstext}
\usepackage{amsopn}
\usepackage{amsmath}
\usepackage[T1]{fontenc}
\usepackage{xcolor}
\definecolor{myblue}{rgb}{0.0, 0.0, 1.0}
\definecolor{mygreen}{rgb}{0.01,0.75,0.20}
\hypersetup{ linkcolor=myblue, colorlinks=true, urlcolor=black, citecolor=red}
\usepackage{hyperref}
\newcommand{\Hmm}[1]{\leavevmode{\marginpar{\tiny%
$\hbox to 0mm{\hspace*{-0.5mm}$\leftarrow$\hss}%
\vcenter{\vrule depth 0.1mm height 0.1mm width \the\marginparwidth}%
\hbox to 0mm{\hss$\rightarrow$\hspace*{-0.5mm}}$\\\relax\raggedright #1}}}

\newcommand{\dt}{\,\mathrm{d}t}

\newtheorem{thm}{Theorem}[section]
\newtheorem{cor}[thm]{Corollary}
\newtheorem{lem}[thm]{Lemma}
\newtheorem{lemma}[thm]{Lemma}

\theoremstyle{definition}

\newtheorem*{rem}{Remark}

\newcommand{\Z}{{\mathbb Z}}
\newcommand{\R}{{\mathbb R}}

\newcommand{\N}{{\mathbb N}}

\newcommand{\ph}{{\varphi}}

\def\ga{\alpha}

            \def\gl{\lambda}

\def\gs{\sigma}

     \def\Gd{\Delta}

\begin{document}
\title[Landis' conjecture on graphs]{On Landis' conjecture for positive Schr\"odinger operators on graphs}

	\author {Ujjal Das}
\address {Ujjal Das, Department of Mathematics, Technion Israel Institute of
	Technology,   Haifa, Israel}
\email {ujjaldas@campus.technion.ac.il, getujjaldas@gmail.com}
\author{Matthias Keller} 
\address{Matthias Keller, Universit\"at Potsdam, Institut f\"ur Mathematik, 14476  Potsdam, Germany}
\email{matthias.keller@uni-potsdam.de}

	\author{Yehuda Pinchover}
\address{Yehuda Pinchover,
	Department of Mathematics, Technion Israel Institute of
	Technology,   Haifa, Israel}
\email{pincho@technion.ac.il}

\date{\today}
%\maketitle
%\tableofcontents

% \frenchspacing
%%%%%%%%%%%%%%%%%%%%%%%%%%%%%%%%%%%%%%%%%%%%%%%%%%%%%%%%%%%%
% ABSTRACT
%%%%%%%%%%%%%%%%%%%%%%%%%%%%%%%%%%%%%%%%%%%%%%%%%%%%%%%%%%%%
\begin{abstract} In this note we study the  Landis conjecture for positive Schr\"odin\-ger operators on graphs. More precisely, we prove a Landis-type result in the form of a decay criterion that ensures when $ \mathcal{H} $-harmonic functions for a positive Schr\"odinger operator $ \mathcal{H} $ with potentials bounded from above by $ 1 $  are trivial. The positivity assumption on the operator allows us to impose slow decay across the entire graph, while requiring fast decay in only one direction, rather than throughout the whole graph. We then specifically look at the special cases of $ \mathbb{Z}^{d} $ and regular trees for which we get a explicit decay criterion. Moreover, we consider the fractional analogue of the  Landis conjecture on $ \mathbb{Z}^{d} $. Our approach relies on the discrete version of Liouville comparison principle which is also proved in this article.
	
%%%%%
	
	
	{\color{blue} \ \ \ \ \ \ To appear in International Mathematics Research Notices}
\end{abstract}

\maketitle 
%{\  }
%\\[-15mm]

\section{Introduction}
A well known conjecture of Landis  says that an $\mathcal{H}$-harmonic function $ u $ of a Schr\"od\-inger operator $ \mathcal{H} = \Delta+V $ on $\mathbb{R}^{d}$, i.e.,
\begin{align*}
	\mathcal{H}u=\Delta u +V u := -\sum_{j=1}^{d}\partial_{j}^{2}u +Vu= 0 \ \ \ \mbox{on} \  \mathbb{R}^{d} 
\end{align*}
 with $ |V|\leq 1 $ which satisfies $$  |u(x)|\leq \exp({-C|x|} ) $$ for sufficiently large $ C>0 $ is trivially zero, \cite{Landis2,FBSR24}. Landis  also conjectured  a weaker version which states that   if  $$  |u(x)|\leq \exp({-|x|^{1+\varepsilon}} ) $$ for $ \varepsilon>0 $, then $ u=0 $ in $\R^d$. Such a result  can be seen as a property of {\it unique continuation at infinity} for $\mathcal{H}$-harmonic functions in $\R^d$.

In 1992, Meshkov \cite{Meshkov} provided a counterexample for a complex valued potential in two dimensions by constructing a {\em complex} valued bounded potential $ V $ and a nontrivial  $\mathcal{H}$-harmonic function $u$ in $\R^2$  which satisfies  $ |u(x)|\leq \exp({-C|x|^{4/3}} )$ for some constant $ C>0 $. Furthermore,  Meshkov showed that if 
\begin{align*}
	 \sup_{x\in \mathbb{R}^{d}} |u(x)| \exp({C|x|^{4/3}} )<\infty
\end{align*}
for all $ C>0 $, then $ u=0 $ in $\R^d$. {The result with exponent $ 2 $ instead of $ 4/3 $ already appears in the earlier work  \cite{FHH2O2} (put $ \beta=1 $ in Theorem 2.3 therein) which however seems to be neglected in the  literature.} {Although, due to {Meshkov's results,} the Landis conjecture is settled for complex-valued bounded potentials, it still remains open in the real-valued case.}

{The purpose of this paper is to study the Landis conjecture in the {\em real valued case} on general discrete graphs under the additional assumption that the Schr\"odinger operator in question is {\em positive} in the sense of the quadratic form which is equivalent to existence of a positive $\mathcal{H}$-superharmonic function. For the precise definition of $\mathcal{H}$ on discrete graphs and for more details on the positivity assumption, see Section~\ref{sec_set_up}. Our approach is inspired by the recent paper \cite{DP24}, where the corresponding problem in the continuum is studied. On the one hand, the positivity of the operator is clearly an essential restriction. However, with the additional positivity assumption, a sharp decay criterion for the validity of the Landis conjecture is obtained on the continuum case in \cite{DP24} and on graphs in the present paper. We note that various of the known results in $\R^d$ implicitly assume the positivity of $\mathcal{H}$ (see for example, \cite{ABG,KSW,Sirakov}).}

{There is a great deal of research on Landis-type conjecture, our bibliography refers only to a small fraction of it. Let us first review the literature that omits the positivity assumption.  In $ \mathbb{R}^{2} $, we should mention the breakthrough result \cite{LMNN} by Logunov, Malinnikova, Nadirashvili, and  Nazarov, which is, to our knowledge, the strongest results available so far. By the strongest result we mean that the authors assume the weakest, compared to the available literature, decay condition on $u$ for the validity of the Landis conjecture. For earlier results on $\R^2$, see the references in \cite{LMNN}. Landis-type results on $\R^d$ without a positivity assumption are proved in \cite{BK}.  In $\mathbb{Z}^d$, for a bounded potential $V$, without the positivity assumption, it has been shown in \cite{LM} that if
$$ {\liminf_{N \rightarrow \infty} \frac{1}{N}\log \max_{|x|_\infty \in [N,N+1]}|u(x)|  <-\|V\|_{\infty}-4d+1} \,,$$
then $u=0$ in $\mathbb{Z}^d$, where $|x|_{\infty}=\max_{k=1,\ldots,d}|x_{k}|$. For related results, see \cite{FV}.
We also refer to  \cite{FBRS24}, where the authors recently studied the Landis-type uniqueness results in a mesh $(h\mathbb{Z})^d$, $h>0$ under certain summation criteria on $u$, and analyzed the behavior of the solutions as the mesh-size $h$ decreases
to zero.}

In \cite{KSW}, authors proved a weaker version of Landis' conjecture in $\R^2$ with positivity assumption on the potential $V$ (which imply that $\mathcal{H}$ is positive). {The papers \cite{ABG,DP24,Sirakov} address the conjecture in $ \mathbb{R}^{d} $  under the positivity assumption, and \cite{Rossi,DP24} study it assuming positivity in {\it exterior domains}. One may note that some of the mentioned articles concern more general operators than $\mathcal{H}$. 

In the Euclidean settings, Landis-type results have been also studied for Dirac operators \cite{Cassano}, {fractional Schr\"odinger operators \cite{RW19,KOW},} the time-dependent Schr\"odinger equation and the heat equations  \cite{EKPV10,EKPV16}, see also the references therein. 
There are results for time-independent and time dependent Schr\"odinger equations on graphs with discrete Laplacian in \cite{BMP24,BP24,FB,FBRS24,JLMP18}.} The case of a half cylinder in $ \mathbb{R}^{d} $ is addressed in \cite{Fil}. {For unique continuation results on manifold, we refer to \cite{PPV} and the references therein. We also refer to a recent review on Landis' conjecture \cite{FBSR24}, which describes the state of the art on the topic and provides a comprehensive list of references.}

In Theorem \ref{t:Landis_general}, we show that $\mathcal{H}$-harmonic functions, i.e., solutions of the equation $\mathcal{H}u \!=\! 0 $ on a discrete graph, are trivial under much weaker decay conditions than those discussed above. Moreover, instead of $ |V|\leq 1 $ as in the Landis conjecture,  we only assume $$ V\leq 1 .$$ While the positivity assumption on the Schr\"odinger operator of course restricts how negative $V$ can be, we do not need a pointwise lower bound on $V$.   

Next, we give an overview of the decay conditions given in Theorem \ref{t:Landis_general} and the overall strategy. For the details and precise definitions, we refer the reader to the next sections.  We consider $\mathcal{H}$-harmonic   functions $ u $ which should satisfy an ``a priori estimate'' on the whole space which is typically not  in $ \ell^{2} $ and even might only be in $\ell^\infty$.  Secondly, a much stronger  ``{decay estimate}'' must be satisfied but only in one direction of the space.  Such  the  Landis an $\mathcal{H}$-harmonic $ u $ is then shown to be trivial.

The ``a priori estimate'' will guarantee  that {$ u $ has a constant sign, and therefore, we may assume that $u$ is positive.} If the function we compare $ u $ to, was in $ \ell^{2} $, then  positivity of $ u $ is trivial because the $\mathcal{H}$-harmonic function $ u $ would then be either trivial or an eigenfunction at the bottom of the spectrum, such an eigenfunction is known to be strictly positive. To go beyond $ \ell^{2} $, we prove and employ a discrete version of a {\it{Liouville comparison principle}}, Theorem~\ref{thm:comparison}, which goes back to \cite{Pinchover} in the continuum. Specifically, the comparison of  the  Landis an $\mathcal{H}$-subharmonic  function $u_+$ will be made with a ``slowly decaying'' {\it{Agmon ground state}} of a related {\it{critical operator}}.

For the ``decay estimate'', we compare $ u $ with $ G_{1} $ which is given as the resolvent of $ \mathcal{H}_1=\Delta+1 $   applied to the delta function at a fixed vertex. This $ G_{1} $ is a positive solution of the equation $ \mathcal{H}_1 \varphi=0  $ of minimal growth at infinity. If we now assume that the potential satisfies $ V\leq 1 $,  then  $ u $ is also a positive supersolution of $ \mathcal{H}_1 \varphi=0 $, and hence, it cannot decay faster than $ G_{1} $ in any direction unless $u = 0$.  For general graphs, this is comprised in Theorem~\ref{t:Landis_general} and Corollaries~\ref{cor:landis1},~\ref{cor:landis2} and~\ref{cor:landis3}. Moreover, Theorem~\ref{t:Landis_general} is sharp, see the remark below the proof of Theorem~\ref{t:Landis_general}.

While for general graphs, we have of course no explicit estimates of the resolvent, on $ \mathbb{Z}^{d} $ and regular trees one has a rather good understanding of $ G_{1} $, \cite{Michta,Molchanov}. Specifically, for  the  Landis an $\mathcal{H}$-harmonic function $ u $ of a positive Schr\"odinger operator $ \mathcal{H} $ on $ \mathbb{Z}^{d} $ with potential $ V\leq 1 $, we can show the following:  For   $ d=1,2 $  the ``a priori estimate'' make us impose the condition that $ u $ is bounded in the whole space, and for $ d\ge 3 $ that $ u $ has to be bounded by the square root of the Laplacian's minimal positive Green function $ G_{0} $ at $ 0 $, which implies that $ u\in O(|x|^{(2-d)/2}) $ near infinity, see e.g. \cite{Woess}. More precisely,  this uses that $ G_{0}^{1/2} $ is an Agmon ground state for $ \Delta-W $ for a specific {\it critical Hardy weight} $ W$, \cite{KellerPinchoverPogorzelski_optimal}. 
 Furthermore, it is known that  one can estimate $ G_{1}(x) $ by a multiple of $ |x|^{(1-d)/2} \mathrm{e}^{-|x|}$, cf. \cite{Molchanov,Michta}. Hence, on $ \mathbb{Z}^{d} $, if   $ u $ satisfies  the a priori estimate  $ u\in O(|x|^{(2-d)/2}) $ and
 \begin{align*}
 	\liminf_{x\to\infty}|u(x)||x|^{(d-1)/2} \mathrm{e}^{|x|}=0,
 \end{align*}
then $ u=0 $ by Theorem~\ref{mainZd}, where $x\to\infty$ is always understood with respect to the one-point compactification of the underlying space. 

On $ d $-regular trees, we obtain that if $ u\in O(|x|^{1/2}d^{-|x|/2}) $ and  
\begin{align*}
	\liminf_{x\to\infty}|u(x)|d^{|x|}  =0,
\end{align*}
then $ u=0 $ by Theorem~\ref{t:trees}. Here we use the considerations for {\it{optimal Hardy weights}} on regular trees by \cite{Berchio}.

We furthermore consider the Landis-type results for Schr\"odinger operator involving the {\em fractional Laplacian} $\Delta^{\sigma}$ on $ \mathbb{Z}^{d} $, $\sigma \in (0,1)$. We use recent estimates on the Green's function derived in \cite{DER24} to show that if $ u\in O \left({|x|^{2\sigma-d}}\right) $ and  $$   \liminf_{x\to\infty} |u(x)| |x|^{2\sigma+d}=0, $$
then $ u \!=\!0 $ $ \mathbb{Z}^{d} $ by Theorem~\ref{mainZdfrac}. Moreover, for $ d=1 $, the a priori bound which is needed can be improved to $  O \left(|x|^{(2\sigma-d)/2}\right) $ due to recent results by \cite{KN23}, see Theorem~\ref{fracd=1}.

The paper is structured as follows. In the next section, we  introduce the set-up and 			prove a 			Liouville comparison theorem, Theorem~\ref{thm:comparison}. In Section~\ref{sec:graphs}, we study the  problem on general graphs and prove the main				abstract result, Theorem~\ref{t:Landis_general} together with its corollaries, 	Corollaries~\ref{cor:landis1},~\ref{cor:landis2} and~\ref{cor:landis3}. In Section~\ref{sec:Examples}, we apply these results to 										$ \mathbb{Z}^{d} $ in Theorem~\ref{mainZd} and~\ref{mainZd_2}, and to regular trees in Theorem~\ref{t:trees}.	Finally, in Section~\ref{sec:Frac} we study the fractional analogue.

\section{Set up and a Liouville comparison theorem}\label{sec_set_up}
In this section, we recall some basic notions and results in criticality theory that are essential for the development of this article, where we use \cite{KellerPinchoverPogorzelski} as the main source for the  discrete setting. 
In addition, we prove a discrete analogue of a Liouville comparison principle for Schr\"odinger operators on graphs.

Throughout the paper, we use the following notation and conventions:

\begin{itemize}
\item For a set $A$, we write $\#A$ to denote its cardinality.

\item For two subsets $A,B$ in a discrete topological space $X$, we write $A \Subset B$ if  $A$  is a compact, i.e., a finite subset of  $B$.

\item If a function $ f \geq 0$ is not trivial, then we say $ f $ is \emph{positive}. Moreover, if $ f>0 $, then we say $ f $ is \emph{strictly positive}. 
\item For a set $A$,  we denote by $1_{A}$ the characteristic function of $A$ and for a singleton set $\{x\}$ we denote $1_{x}=1_{\{x\}}$. 

\item For two functions $ f, g:A\to\R $, we denote $ f \vee g $ the pointwise maximum of $f$ and $g$.

\item For positive functions $ f, g:A\to\R $, we write $ f\asymp g $ if there is a constant $ C>0 $ such that $ C^{-1}f \leq g\leq C f$.

\item For two functions $ f,g :X\to \R $ such that $ g $ does not vanish outside of a compact set,  we denote 
\begin{align*}
f\in O(g)\qquad\Longleftrightarrow\qquad \limsup_{x\to \infty}\frac{|f(x)|}{|g(x)|}<\infty,
\end{align*}
where the limit $ x\to\infty $ is taken in the topology of the one-point compactification $ X\cup\{\infty\} $ of $ X $.
\end{itemize}

Let $ X $ be an infinite countable set equipped with the discrete topology. We denote by $ C(X) $ the space of real valued functions on $ X $ and by $ C_{c}(X) $ the subspace of functions of compact (i.e. finite) support. For a function $ f:X\to \R $, we write $$  \sum_{X}f=\sum_{x\in X}f(x),  $$
whenever the sum is absolutely convergent.
Any strictly positive function $ m:X\to (0,\infty) $ extends to a measure of full support on $ X $ via $ m(A)=\sum_{x\in A}m(x) $.

A {\em graph} $ b $ over the measure space $ (X,m) $ is a symmetric function $ b:X\times X\to [0,\infty) $ with zero diagonal which is locally summable, i.e.,
\begin{align*}
	\sum_{y\in X}b(x,y)<\infty,\qquad { x\in X}.
\end{align*}
Given a graph $ b $, we denote $ x\sim y $ whenever $ b(x,y)>0 $ and think of $ x $ and $ y $ to be connected by an edge. The {\em (vertex) degree} of $x \in X$ is the number of vertices connected to
$x$. 

{Throughout this paper, we assume that the graph $b$ is \emph{connected}}, i.e., for every $ x,y\in X $ there are vertices $ x=x_{0}\sim \ldots \sim x_{n}=y $ connecting $ x $ and $ y $ by a path.

We introduce the (positive) Laplacian $ \Delta =\Delta_{b,m}$  acting 
on its formal domain $ \mathcal{F}=\mathcal{F}_{b} $
$$  \mathcal{F}=\{ f:X\to \mathbb{R}\mid \sum_{y\in X}b(x,y)|f(y)|<\infty\mbox{ for all }x\in X\}  $$ 
as
\begin{align*}
	\Delta f(x)=\frac{1}{m(x)}\sum_{y\in X}b(x,y)(f(x)-f(y)).
\end{align*} 
For  a {\em potential} $ V:X\to \R $, we denote the corresponding Schr\"odinger operator 
$ \mathcal{H}=\mathcal{H}_{b,V,m} $ acting on $ f\in \mathcal{F} $ as 
\begin{align*}
	\mathcal{H}f(x)=\Delta f(x)+ {V(x)}f(x).
\end{align*}

%%%%%%%%%%
A  function $u$ is $\mathcal{H}$-$({sub}/{super})harmonic$ on $X$
if $u\in \mathcal{F}$ and $\mathcal{H}u=0$ ($\mathcal{H}u {\leq}  0$ / $\mathcal{H}u {\geq}  0$) on $X$.
We write $$ \mathcal{H}\geq 0 \qquad   \mbox{in }  X $$
if there exists a positive $\mathcal{H}$-superharmonic function $u$ on $X$, and in this case we say that the corresponding Schr\"odinger operator $\mathcal{H}$ is {\it{positive} on $X$}.
\begin{rem}
We recall that in the continuum the $\mathcal{H}$-{(sub/super)harmonic} functions in $\R^d$ are defined in the weak sense i.e., $u \in W^{1,2}_{\mathrm{loc}}(\R^d)$ is called an $\mathcal{H}$-{(sub/super) harmonic} function if 
$$\langle\mathcal{H}u,\varphi \rangle =\int_{\R^d} \nabla u \nabla \varphi \ {\mathrm{d}}x + \int_{\R^d} Vu  \varphi \ {\mathrm{d}}x = 0\quad (\mbox{resp.} \le0 \mbox{ or }\ge 0)  $$
for all  $  \varphi \in C_c^{\infty}(\R^d) $ (with $ \varphi\ge 0 $.)  However, in the discrete setting weakly $\mathcal{H}$-harmonic functions are already $\mathcal{H}$-harmonic functions, see \cite[Theorem~2.2 and Corollary~2.3]{HKLW}.
Note that, for consistency with the discrete settings, we consider the (positive) Laplacian on $\R^d$ as $\Delta_{\R^d}:= -\sum_{j=1}^{d}\partial_{j}^{2}$. 
\end{rem}
By Green's formula \cite[Lemma~2.1]{KellerPinchoverPogorzelski}, $ \mathcal{H} $ is related to its corresponding energy form $ \mathcal{Q}=\mathcal{Q}_{b,V,m} $ 
\begin{align*}
 \mathcal{Q}(\ph) =\frac{1}{2}\sum_{X\times X}b |\nabla \ph|^{2}+\sum_{X}m V\ph^{2},
\end{align*}
which takes finite values for $\ph\in  C_{c}(X) $ and where $ \nabla_{xy} \ph=\ph(x)-\ph(y) $. 
Unless  stated otherwise, we assume  throughout the paper that $ \mathcal{Q} $ is {\em positive} on $ C_{c}(X) $, i.e.,
\begin{align*}
	\mathcal{Q}(\ph)\ge 0\qquad \mbox{for all } \ph\in C_{c}(X).
\end{align*}
{By the Agmon-Allegretto-Piepenbrink-type theorem \cite[Theorem~4.2]{KellerPinchoverPogorzelski}, $\mathcal{Q}(\ph)\!\geq \!0$ on $C_{c}(X)$ if and only if $\mathcal{H}\geq 0$ in $X$.}

{A  positive function $ \phi \!\in\! \mathcal{F}$ is called  a \emph{positive solution  of minimal growth at infinity   in $X$} for $ \mathcal{H} $,  if $ \phi>0 $ and $\mathcal{H}\phi = 0$   in  $X\setminus K_{0}$ for some compact $ K_{0}\Subset X $, and for any positive $ \psi \in \mathcal{F}$ such that $ \psi>0 $ and  ${\mathcal{H}\psi\geq 0}$ on $X\setminus K$  for  $ K_{0}\Subset K\Subset X$ which satisfies~$\phi\leq \psi$ on	$K$, one has $\phi\leq \psi$ in~$X\setminus K$. } 

{A positive $\mathcal{H}$-harmonic function $\phi$ in $X$ which is a solution  of minimal growth at infinity in $X$ is called an {\em Agmon ground state} of $\mathcal{H}$.}

We call a positive Schr\"odinger operator $\mathcal{H}$ \emph{subcritical} if there is a non-trivial $W \ge 0$ such that for all $\ph\in C_c(X)$
\begin{align*}
\mathcal{Q}(\ph)\ge \sum_{X}m W \ph^{2} .
\end{align*}
We call such a function $ W $  a \emph{Hardy weight} for $ \mathcal{H} $. If $\mathcal{H}\geq 0 $ is not subcritical, then we say it is \emph{critical}.
Furthermore, we say that a Hardy weight $ W $ is a \emph{critical Hardy weight} for $ \mathcal{H} $ if $ \mathcal{H}-W $ is critical. If the graph $ b $ is connected and the potential $ V$ vanishes, i.e., $ V=0 $, then one also says the graph  is \emph{transient} if $ \mathcal{H} $ is subcritical and \emph{recurrent} if $ \mathcal{H} $ is critical, cf. \cite[Theorem~6.1 and Definition~6.2]{KellerBook}.

Recall that the operator $\mathcal{H}$ is critical if and only if there is $x\in X$ and functions $\ph_{n}\in C_{c}(X)$ with $\ph_{n} \!\ge \!0$, $C^{-1} \!\le\! \ph_{n}(x) \!\le\! C$ for some $ C>0 $ and all $n\geq1$, such that
\begin{align*}
    \mathcal{Q}(\ph_{n})\to0\quad \mbox{as } n\to\infty,
\end{align*}
cf. \cite[Theorem~5.3]{KellerPinchoverPogorzelski}.
We call such a sequence $(\ph_{n})$  a \emph{null-sequence} for $ \mathcal{H} $ with respect to the vertex $x\in X$. Observe that such a null-sequence converges to an Agmon ground state of $ \mathcal{H} $, \cite[Theorem~5.3]{KellerPinchoverPogorzelski}. A further characterization of criticality of $\mathcal{H}$ is that up to scalar multiples there is a unique positive $\mathcal{H}$-superharmonic function which is then $\mathcal{H}$-harmonic and an Agmon ground state, see e.g.  \cite[Theorem~5.3]{KellerPinchoverPogorzelski}.

Finally, we recall  the {\it{ground state transform}}, cf. \cite[Proposition~4.8]{KellerPinchoverPogorzelski}. {Let 
$ f\in \mathcal{F} $ be a positive function, then for all $\ph\in C_{c}(X)$} 
\begin{align*}
    \mathcal{Q}(f\ph) =\frac{1}{2}\sum_{X\times X}b (f\otimes f)|\nabla \ph|^{2}+\sum_{X}m(f\mathcal{H}f)\ph^{2} \,,
\end{align*}
where $f\otimes f :X\times X\to [0,\infty) $ is defined as $f\otimes f(x,y)=f(x)f(y)$.

We now prove {a Liouville} comparison theorem which is vital for the considerations of this paper, cf. \cite{Pinchover}.

\begin{thm}[{Liouville} comparison principle]\label{thm:comparison} Suppose $\mathcal{H}$ and $\mathcal{H}'$ are positive Schr\"odinger operators associated to connected graphs $b$, $b'$ and potentials $V, V' $ over $ (X,m) $. Let $u\in \mathcal{F}_b$ and $v\in \mathcal{F}_{b'}$ be such that
\begin{itemize}
  \item [(a)] $\mathcal{H}'$ is critical and $v>0$ is its Agmon ground state,
  \item [(b)] $u_{+}\neq 0$  and $\mathcal{H}u_{+}\leq0$,
  \item [(c)] $ b(u_{+}\otimes u_{+})\leq C  b'(v\otimes v)$ for some $ C>0 $.
\end{itemize}
Then, $\mathcal{H}$ is critical and $u>0$ is an Agmon ground state of $\mathcal{H}$.
\end{thm}

\begin{proof} By (a) there exists a null-sequence $(\ph_{k}) \subset C_{c}(X)$ with respect to a vertex $x\in X$ for the energy function $ \mathcal{Q}' $ of $\mathcal{H}'$. For $ k\in\mathbb{N} $, set
\begin{align*}
    \psi_{k}=\frac{u_{+}\ph_{k}}{v} \, .
\end{align*}
Observe that the ground state transform of the energy functional $\mathcal{Q}$ of $\mathcal{H}$ with respect to $u_+$, (b) and (c) yields 
\begin{align*}
	\mathcal{Q}(\psi_{k})&= \frac{1}{2}\sum_{X\times X}b (u_{+}\otimes u_{+})|\nabla (\ph_{k}/v)|^{2}+\sum_{X}m(u_{+}\mathcal{H}u_{+})({\ph_{k}}/{v})^{2}\\&\le \frac{C}{2}\sum_{X\times X}b' (v\otimes v)|\nabla (\ph_{k}/v)|^{2}\\
	&= C	\mathcal{Q}'(\ph_{k})\to 0,\qquad k\to\infty.
\end{align*}
Thus, $( \psi_{k} )$ is a null-sequence for $ \mathcal{H} $, and $ \mathcal{H} $ is critical. In particular, 
$( \psi_{k} )$ converges pointwise to an Agmon ground state for $ \mathcal{H} $. Since $ \mathcal{H}' $ is critical and $ v>0 $ is $ \mathcal{H} $-harmonic, the null-sequence $ (\ph_{k}) $ converges pointwise to a positive multiple of $ v>0$. Hence, $ (\psi_k) $ converges pointwise to a positive multiple of $ u_{+} $ which is therefore strictly positive.  Thus, $u_+=u$ is an Agmon ground state for $ \mathcal{H} $. 
\end{proof}

\section{General Landis type theorem}\label{sec:graphs}
In the present section we prove the main abstract Landis-type theorem and derive some of its important corollaries which will play crucial role in the subsequent section concerning the cases of the Euclidean lattice and regular trees.

Let $ \mathcal{H} $ be a Schr\"odinger operator associated to a connected graph $ b $ over $ (X,m) $ with potential $V $. {We assume that $\mathcal{H}$ is positive which is equivalent to the fact that $\mathcal{Q}(\ph)\geq 0$ for all $\ph\in C_c(X)$  as discussed above}.

Since the energy functional of $ \Delta  $ is positive, the operator $ (\Delta+\alpha) $ is subcritical for $ \alpha>0 $. Hence, for every $ o\in X $ there exists a  Green function $ G_{\alpha} $ which is the smallest strictly positive function $ \phi $ such that $ (\Delta+\alpha )\phi\ge 1_{o} $, cf. \cite[Theorem~5.16]{KellerPinchoverPogorzelski}, where $1_{o}$ denotes the characteristic function of the vertex $o$. Moreover, the Green function satisfies
\begin{align*}
	(\Delta+\alpha )G_{\alpha}=1_{o}
\end{align*} 
and is a positive solution of minimal growth at infinity for $ (\Delta+\alpha) $ {\cite[Theorem~2.5]{F24}.}  Let $ L $ be the positive selfadjoint operator associated to the form closure  of  $ \mathcal{Q}_{0}\vert_{C_{c}(X)} $ in $ \ell^{2}(X,m) $ where $ \mathcal{Q}_{0} $ is the energy functional of $  \Delta $. Then, also by \cite[Theorem~5.16]{KellerPinchoverPogorzelski}
\begin{align*}
	G_{\alpha}=(L+\alpha)^{-1}1_{o} \, .
\end{align*}
Moreover, the following limit 
\begin{align*}
	G_{0}:=\lim_{\alpha\searrow\, 0}(L+\alpha)^{-1}1_{o}
\end{align*}
exists pointwise but may take the value $ +\infty $. If $ \Delta $ is subcritical, then $ G_{0} $ takes only finite (strictly positive) values on $ X $ since we assumed the graph to be connected \cite{KellerBook}.

The next theorem is the abstract main result of the paper. 
\begin{thm}\label{t:Landis_general}
{Let $ \mathcal{H}=\Gd+V$ be a positive Schr\"odinger operator associated to a connected graph $ b $ over $ (X,m) $ with potential $ V\le 1 $. Let $ \mathcal{H}'$ be a critical Schr\"odinger operator with Agmon ground state $ v>0 $ associated to a connected graph $ b'$ over $ (X,m)$.  
Any $ \mathcal{H} $-harmonic function $ u $	satisfying
\begin{align*}
	 b(|u|\otimes |u|)\leq C  b'(v\otimes v) \ \mbox{ for some }  C>0 \quad \mbox{and} \quad \liminf_{x\to\infty} \frac{|u(x)|}{G_{1}(x)}=0
\end{align*} is trivially zero.}
\end{thm}
For the proof, we need the following lemma which is found for $ V\ge 0 $ in \cite[Lemma~1.9]{KellerBook}. Although the proof carries over verbatim to the case of general Schr\"odinger operators, we include it here for the convenience of the reader.

\begin{lemma}\label{lem:v_+} The pointwise maximum of two $ \mathcal{H} $-subharmonic functions is $\mathcal{H}$-subharmonic.
\end{lemma}
\begin{proof}
	Let $ u,v $ be $ \mathcal{H} $-subharmonic and $ w=u\vee  v $. Then, for $ x\in X $, we assume $ w(x)=u(x)\ge v(x) $ and we have by case distinction for all $ y\in X $
	\begin{align*}
		\nabla_{xy} w \leq \nabla_{xy} u  .
	\end{align*}
	This in combination with the assumption $w(x)=u(x) $ for this particular $ x $ yields
	\begin{align*}
		\mathcal{H}w (x)= \Delta w(x)+{V(x)}w(x) \leq \Delta u(x)  +{V(x)}u(x)=\mathcal{H}u(x)\leq 0.
	\end{align*}
For $   w(x)=v(x)\ge u(x) $, we obtain analogously $ \mathcal{H}w(x)\leq \mathcal{H}v(x)\leq0 $.
	This settles the claim.
\end{proof}

\begin{proof}[{Proof of Theorem~\ref{t:Landis_general}}]
Let $ u $ be a $ \mathcal{H} $-harmonic function, and assume without loss of generality that $ u_{+}\neq 0 $ (otherwise, consider $ -u$). Then, by  Lemma~\ref{lem:v_+} above, we have
\begin{align*}
	\mathcal{H}u_{+}\le 0.
\end{align*}
We collect the following facts:
\begin{itemize}
	\item [(a)] $\mathcal{H}'$ is critical and $v>0$ is an Agmon ground state.
	\item [(b)] $u_{+}\neq  0$  and $\mathcal{H}u_{+}\leq 0$.
	\item [(c)] {$ b(u_{+}\otimes u_{+})\leq C b'(v\otimes v)$ for some $ C>0 $.} 
\end{itemize}
Thus, by the Liouville comparison theorem, Theorem~\ref{thm:comparison}, we infer that $ \mathcal{H} $ is critical and $ u=u_{+}>0 $ is an Agmon ground state for $ \mathcal{H} $. In particular, $\mathcal{H} u=0$.

Now, since $ V
\leq 1 $ we have 
\begin{align*}
(\Delta+1) u \ge (\Delta+V) u=\mathcal{H} u=0.
\end{align*}
As $ G_{1} $ is a solution of minimal growth at infinity for $ \Delta+1 $, we obtain that $G_1\leq Cu$ in $X$ for some $C>0$ which contradicts our assumption  $\liminf_{x\to\infty} {u(x)}/{G_{1}(x)}=0 $.  Thus, any such $ \mathcal{H} $-harmonic function is trivially zero.
\end{proof}
\begin{rem}
Let us note that Theorem \ref{t:Landis_general} is sharp in the following sense. For $ V=1- [1/G_{1}(o)]1_{o}\leq 1$ with $o\in X$,  the function $ u=G_{1} $ solves 
	\begin{align*}
		(\Delta+V)u= 	(\Delta+1)G_{1}- (1_{o}/G_{1}(o))G_{1}=0.
	\end{align*} 
In view of the decay conditions assumed on $u$ in Theorem \ref{t:Landis_general}, we see the sharpness of our result. 
\end{rem}
%%%%%%%%%%%%%%%
\begin{cor}\label{cor:landis1}
	Let $ \mathcal{H}=\Gd+V $ be a positive Schr\"odinger operator associated to a connected graph $ b $ over $ (X,m) $ with potential $ V\le 1$.  Let  $ W\ge 0 $ be such that  $\Gd-W$ is critical and let $v$ be its Agmon ground state.  If  $u\in O(v) $ is a $ \mathcal{H} $-harmonic function such that
	\begin{align*}
		\liminf_{x\to\infty} \frac{|u(x)|}{G_1(x)}=0,
	\end{align*}
	then $ u=0 $.
\end{cor}
\begin{proof}
The statement follows directly from the theorem above.
\end{proof}

Next, we present two corollaries where the Agmon ground state $ v $ associated to a critical Hardy weight $W$ in the above corollary can be replaced by certain functions of the Green function of the Laplacian.

For the first corollary, recall that a function $ \phi: X\to (0,\infty)$ is said to be \emph{proper} if for any compact set $ I \subset (0,\infty) $ we have that 	$ \phi^{-1}(I)\subseteq X $ is compact, i.e., finite. Furthermore, we say that $ \phi$ is of \emph{bounded oscillation} if
\begin{align*}
	\sup_{x\sim y}\frac{\phi(x)}{\phi(y)}<\infty.
\end{align*}
In \cite[Theorem~1.1]{KellerPinchoverPogorzelski_optimal} it is shown for $\mathcal{H}=\Gd$ that a strictly positive $ \mathcal{H} $-superharmonic function $ \phi $
which is proper, of bounded oscillation and $ \mathcal{H} $-harmonic outside of a compact set gives rise to a critical Hardy weight
\begin{align*}
	W= \frac{\mathcal{H}\phi^{1/2}}{\phi^{1/2}}
\end{align*}
for $\mathcal{H}$. Indeed, $ W$ is even null-critical (in other words, $W$ is an optimal Hardy weight in the sense of \cite{KellerPinchoverPogorzelski_optimal}), i.e., the Agmon ground state ${v= \phi^{1/2}}>0 $ of the critical operator $ \mathcal{H}-W $ is not in $ \ell^{2}(X,Wm) $. 

Observe that whenever there is a strictly positive proper function of bounded oscillation, then the graph must be \emph{locally finite}, i.e., $ \#\{y\in X\mid x\sim y\}<\infty $ for all $ x\in X $. 

The next corollary allows under a properness and bounded oscillation assumptions on $ G_{0} $ to take  ${v}=G_{0}^{1/2}$ in the above {corollary}.
\begin{cor}\label{cor:landis2}
	Let $ \mathcal{H}\!=\!\Gd+V$ be a positive Schr\"odinger operator on the connected graph $ b $ over $ (X,m) $ with potential $ V \!\le\! 1$. Assume further that $ \Delta  $ is subcritical and  $ G_{0} $, the minimal positive Green function of $\Delta$,  is proper and of bounded oscillation.  Let  $ u \!\in\! O(G_{0}^{1/2}) $ be  a $ \mathcal{H} $-harmonic function satisfying 
	\begin{align*}
		\liminf_{x\to\infty} \frac{|u(x)|}{G_{1}(x)}=0.
	\end{align*}
Then $ u	=0 $.
\end{cor}
\begin{proof}
	By the supersolution construction \cite[Theorem~1.1]{KellerPinchoverPogorzelski_optimal}, the function
	\begin{align*}
		W=\frac{\Delta G_{0}^{1/2}}{G_{0}^{1/2}}
	\end{align*}
	is a critical Hardy weight for $ \Delta $ and the associated Agmon ground state of $\Gd-W$ is  $v= G_{0}^{1/2} $. {Hence, the statement follows now directly from Corollary ~\ref{cor:landis1} above.}
\end{proof}

Next, we give a corollary, when no explicit critical Hardy weight is available. In this case, we make a stronger decay assumption on $u$ comparing to the corresponding one in the previous corollary.

\begin{cor}\label{cor:landis3}
	Let $ \mathcal{H}=\Gd+V $ be a positive Schr\"odinger operator associated to a connected graph $ b $ over $ (X,m) $ with potential $ V\le 1$.   Assume that  $ (\Delta +\alpha)$   is subcritical for some $ \alpha\in\R$ and $u$ is  a $ \mathcal{H} $-harmonic function  satisfying $u\in O(G_{\alpha}) $, and  
	\begin{align*}
		\liminf_{x\to\infty} \frac{|u(x)|}{G_1(x)}=0.
	\end{align*}
	Then $ u=0 $.
\end{cor}
\begin{proof}
Fix such an $\ga$. Let $ o\in X $ and choose $ C_{o}>0 $ such that $\mathcal{H}':= \Delta+\alpha-C_{o}1_{o} $ is critical (cf. \cite[Lemma~4.4]{Das_etal}) with an Agmon ground state $ v>0 $. {Clearly, $ O(v)=O(G_{\alpha})$ since $ G_{\ga}$ is a solution of minimal growth at infinity for $\Delta+\alpha$.  Hence, $ u\in O(G_{\alpha}) = O(v)$.} Thus, the statement follows from Theorem~\ref{t:Landis_general}. 
\end{proof}

Finally, we formulate a corollary to relate our result with the summability criteria given in \cite{FBRS24}.

\begin{cor} \label{Cor:FBR}	Let $ \mathcal{H} $ be a positive Schr\"odinger operator associated to a connected graph $ b $ over $ (X,m) $ with potential $ V\le 1$. Every $ \mathcal{H} $-harmonic function $ u\in \ell^{2}(X, G_{1}^{-2}) $ 
is trivial.
\end{cor}
\begin{proof}
	The summability of $ G_{1}^{-2}|u|^{2} $ implies $u\in  o(G_{1}) $, i.e. $ \lim_{x\to\infty}|u(x)|/G_{1}(x)=0 $. Thus, the statement follows from Corollary~\ref{cor:landis3} above.
\end{proof}

\begin{rem}[Exterior domains]
It is not hard to derive the Landis-type results on exterior domains, i.e., on sets $ \tilde X=X\setminus K $ for $ K\Subset X $. Let $ b $ be a graph over $ (X,m) $ and $ V $ be a given potential. Assume the form $ {\mathcal{Q}} $ is positive on $ C_{c}(\tilde X) $. Then, the operator acting as 
\begin{align*}
	\tilde{\mathcal{H}}= \tilde\Delta+D+ V_{\tilde X}
\end{align*}
where $\tilde \Delta=\Delta_{b\vert_{\tilde X\times \tilde X},m\vert_{\tilde X}}$, $ V_{\tilde X}=V\vert_{\tilde X} $ and $ D: \tilde X\to[0,\infty) $
\begin{align*}
	D(x) = \frac{1}{m(x)}\sum_{y\in K} b(x,y)
\end{align*}
is positive on $\tilde{X}$. Note that, the restriction of the graph $ b $ to $\tilde X \!\times\! \tilde X $ is not necessarily connected. However, one can deal with every connected component separately.  Furthermore, if $ K $ is non-empty, then $ D $ does not vanish identically on any connected component of $ \tilde X $ since the original graph is connected. Thus, the operator $\tilde\Delta + D  $ is always subcritical and we denote the Green's function of $ \tilde \Delta + D  + \alpha $ for $ \alpha \ge 0 $ by $ \tilde G_{\alpha} $. Then  one can reproduce all results for $ \tilde {\mathcal{H} }$-harmonic functions on $ \tilde X $  by replacing $ G_{\alpha} $ with $ \tilde G_{\alpha} $. Furthermore, note that if the graph is locally finite, then $ D $ is finitely supported. In this case, every positive solution of minimal growth at infinity  in $ \tilde X $  is comparable to a positive solution of minimal growth at infinity in $ X $.\\
In a similar fashion, we can replace $ \Delta $ on $ X $ by an arbitrary operator $ \Delta+D $ on $ X $ for a positive potential $ D:X\to[0,\infty) $ and obtain  analogous results.
\end{rem}

\begin{rem}
	{In \cite[Section~6]{KSW} the authors suggest a strategy to remove the positivity assumption on a Schr\"odinger operator $\mathcal{H}$ with a bounded real valued potential $V$. Given  the  Landis an $\mathcal{H}$-harmonic $u$ on $\mathbb{R}^d $, they add a dimension and consider a separated solution $w(x,x_{d+1})= u(x)v(x_{d+1})$ on $\mathbb{R}^{d+1}=\mathbb{R}^d \times \R$ satisfying 
	$$\tilde{\mathcal{H}}w=(\Delta_{\R^{d+1}}+V+\lambda^{2})w=0\qquad \mbox{on } \R^{d+1},$$
	where $\gl$ is large enough constant such that $\tilde{\mathcal{H}}$ is a positive operator and $v$ is an explicit one-dimensional function satisfying $v''+\lambda^{2} v=0$. }
	
	{An analogous strategy works on $\Z^d$ as well. One may anticipate proving a Landis-type result for a (not necessarily positive) Schr\"odinger operator $\Delta_{\mathbb{Z}^d} +V$ on $\mathbb{Z}^d $ with a bounded real valued potential $V$, by applying our approach to the positive Schr\"odinger operator $\Delta_{\Z^{d+1}} +V+\lambda^{2}$ on $\Z^{d+1}$ with its separated solution $w$. }
		
		{However, in the first step of the proof Theorem~\ref{t:Landis_general}, we use the Liouville comparison principle to ensure that $w>0$. If we could prove that $w>0$, it would follow that $u$ has a definite sign. In light of the Agmon-Allegretto-Piepenbrink theorem this would imply  that the original operator $\Delta +V$ is, in fact, a positive operator on $\mathbb{Z}^d $.}
	\end{rem}
%%%%%%%%%%%%%%%%%%%%%%%%%%%
\section{The Euclidean lattice and regular trees}\label{sec:Examples}
In this section we apply the abstract results of the previous section
to graphs where the asymptotics of the Green functions for the relevant Schr\"odinger operators are known. We first consider the Euclidean lattice and then the case of trees.

\subsection{Euclidean lattice}

In the continuum setting of $ \mathbb{R}^{d} $ the Green function $ G_{\mathbb{R}^{d},1} $ of  $ \Delta_{\mathbb{R}^{d}} +1$, where $ \Delta_{\mathbb{R}^{d}}=-\sum_{j=1}^{d}\partial_{j}^{2} $ has the asymptotics \cite[Appendix]{DP24}
\begin{align*}
	G_{\mathbb{R}^{d},1}(x) \asymp |x|^{(1-d)/2}\mathrm{e}^{-|x|} \qquad \mbox{as } x\to\infty,
\end{align*}
For the Euclidean lattice, the situation is substantially  more complicated as $ \mathbb{Z}^{d} $ lacks the spherical symmetry, and therefore does not allow for the reduction to a  one-dimensional problem. However, the asymptotics of {$G_1$} are still rather well understood by now \cite{Michta,Molchanov} and we will take advantage of these results.

To be more precise let $ X=\mathbb{Z}^{d} $. The Laplacian $ \Delta $ with standard weights, i.e.,  $ b(x,y)=1$ if $|x-y|=1 $ and 0 otherwise for $ x,y\in \mathbb{Z}^{d} $ and $ m= 1 $, acts as
\begin{align*}
	\Delta f(x)=\sum_{|y-x|=1}(f(x)-f(y))
\end{align*}  
on all functions $f:X\to\mathbb{R}  $. Furthermore, we denote the Euclidean norm on $ \mathbb{Z}^{d} $ by $ |x|=(|x_{1}|^{2}+\ldots +|x_{d}|^{2})^{1/2} $. With slight abuse of notation we also write $ x $ for the identity function  on $ \mathbb{Z}^{^d} $  and $ x_{i} $ for the projection on the $ i $-th component, $ i=1,\ldots,d $.

We prove the  following theorem which is a direct consequence of Theorem~\ref{t:Landis_general} above and the asymptotics obtained in \cite{Michta,Molchanov}.

\begin{thm}\label{mainZd} Let  $u$ be a  $ \mathcal{H} $-harmonic function of a positive Schr\"odinger operator $ \mathcal{H} =\Delta+V $ on $\mathbb{Z}^{d}  $ with $ V\leq 1 $. If 
\begin{itemize}
	\item [(a)] $ u $ is bounded for $ d=1,2 $, 
	\item [(b)] $ u $ satisfies $  u\in O \left(|x|^{(2-d)/2}\right) $ for $ d\ge3 $,
 \end{itemize}	
	and 
	\begin{align*}
		    \liminf_{x\to\infty} |u(x)| |x|^{(d-1)/2}\mathrm{e}^{|x|}=0,  
	\end{align*}
then $ u=0 $.
\end{thm}
For the proof, we discuss the decay of $ G_{\alpha} $ on $\mathbb{Z}^d$ which was studied in the literature. 	In \cite{Michta} it is shown that the function
	$$  C_{a} = \left(\frac{1}{2d}\Delta +a^{2}\right)  ^{-1}1_{0} \qquad \mbox{on } \   \mathbb{Z}^{d}, \ a \in \R $$
	 has the asymptotics
	$      {m_{a}^{(d-3)/2}}{|x|_{a}^{-(d-1)/2}}\mathrm{e}^{-m_{a}|x|_{a}}  $,
	where
	$   m_a=\cosh^{-1}\left(1+da^{2}\right) $
	and
	$$   |x|_a :=\frac{1}{m_a}\sum_{i=1}^{d}x_{i}\sinh^{-1}\left(x_{i}r(x)\right) $$
	with $ r(x) $ being the unique solution to
	$$   \frac{1}{d}\sum_{i=1}^{d}\sqrt{1+x_{i}^{2}r(x)^{2}}=1+a^{2}. $$
	Note that a unique solution exists as the left hand side is strictly monotone in $ r\ge 0 $.
	In \cite[Theorem 3.2]{Molchanov} it is shown with a somewhat different notation that
	\begin{align*} 	
		 C_{a}(x)= \frac{m_{a}^{(d-3)/2}}{{|x|_{a}^{(d-1)/2}}}\mathrm{e}^{-m_{a}|x|_{a}}(1+o(1)).
	\end{align*}
	The reason we cite   \cite[Theorem~3.2]{Molchanov} for the asymptotics over  \cite[Theorem~1.3]{Michta} is that the  authors of \cite{Michta} consider directional asymptotics $ nx $ for $ n\to\infty $ for fixed $ x $ and variable $ a $ rather than estimates which hold for all $ x $ but for fixed $ a $. However, the benefit of the considerations of \cite{Michta} is that they identify $ |\cdot|_{a} $ as a norm. 	To relate the results of  \cite{Molchanov} and  \cite{Michta} to our situation,
	we  observe that for $ \alpha>0 $ 
	$$   G_{\alpha}=  \left(\Delta +\alpha\right)^{-1}1_{0} =\frac{1}{2d} \left(\frac{1}{2d}\Delta +\frac{\alpha}{2d}\right)^{-1}1_{0} =\frac{1}{2d} C_{a}$$
	with
	$   a^{2} ={\alpha}/{2d}.$
	
\begin{proof}[{Proof of Theorem~\ref{mainZd}}]
	For  $ a^2<2 $, we have with  $ a^{2} ={\alpha}/({2d}) $ that $ \alpha< 4d $. Thus, we have, confer Lemma~\ref{l:estimate} and Lemma~\ref{l:estimate0}, 
\begin{align*}
	m_{a}|x|_{a}\leq \sqrt{2a^{2}d} |x|=\sqrt{\alpha} |x|\leq \sqrt{\alpha}|x|_{a} .
\end{align*}
Therefore, we obtain for $a^{2}=1/(2d)  $, i.e., $ \alpha=1 $,
\begin{align*}
 {|x|^{(d-1)/2}}\mathrm{e}^{|x|} \geq C{|x|_{a}^{(d-1)/2}}\mathrm{e}^{m_{a}|x|_{a}} \geq  C C_{a}(x)^{-1}=CG_{1}(x)^{-1}. 
\end{align*}
Hence, the assumption implies
\begin{align*}
		\liminf_{x\to\infty}\frac{|u(x)|}{G_{1}(x)}=0.
\end{align*}
For  $ d=1,2 $, to conclude the statement we apply Corollary~\ref{cor:landis1}. 
In this case, it is well known that the operator $ \Delta $ is critical (i.e, recurrent) and, therefore, $ v=1 $ is an Agmon ground state. Thus, if $ u $ is bounded, then  $ u\in O(v) $. Hence, $ u=0 $ by Corollary~\ref{cor:landis1}.

For $ d\ge 3 $, one knows that the Green function $G_0$ at $ \alpha=0 $ satisfies $ G_{0}\in {O(|x|^{2-d})} $, see e.g. \cite[Theorem~25.11]{Woess}, which is proper and of bounded oscillation. Thus, our assumption $ u\in O({|x|^{(2-d)/2}}) $ ensures that $ u\in O(G_{0}^{1/2})$. Hence, Corollary~\ref{cor:landis2} implies 
$ u=0$.
\end{proof}

We add another variant of the theorem above. This time we focus on the decay on the axis.

\begin{thm}\label{mainZd_2} Let  $u$ be an  $ \mathcal{H} $-harmonic function of a positive Schr\"odinger operator $ \mathcal{H} =\Delta+V $ on $\mathbb{Z}^{d}  $ with $ V\leq 1 $, and let  $\lambda=\cosh^{-1}(3/2)=0.962...<1$. If 
	\begin{itemize}
		\item [(a)] $ u $ is bounded for $ d=1,2 $, 
		\item [(b)] $ u $ satisfies $  u\in O \left({|x|^{(2-d)/2}}\right) $ for $ d\ge3 $,
	\end{itemize}	
	and for an element $ e_{j} $, $ j=1,\ldots,d $ of the standard basis of $ \mathbb{Z}^{d} $ 
	\begin{align*}
		\liminf_{n\to\infty} |u(ne_{j})| n^{(d-1)/2}\mathrm{e}^{\lambda n}=0 ,
	\end{align*}
	then $ u =0$.
\end{thm}
\begin{proof}
	By Lemma~\ref{l:estimate},  for $ x=ne_{j} $,  we have that  $ |x|_{a} =|x_{j}|=|n| $, and $ m_{a}=\cosh^{-1}(3/2) $ for $ a^{2}=1/(2d) $. Hence, $ G_{1}(ne_{j}) = C\cdot n^{(1-d)/2}\mathrm{e}^{-\lambda n} (1+o(1))$.
	 Thus, the result follows along the lines of the proof of  Theorem~\ref{mainZd}.
\end{proof}

\begin{rem} 
{Let us compare Theorem \ref{mainZd} to the results in \cite{FBRS24} and  \cite{LM} assuming that $\mathcal{H}$ is positive. Recall that the results in \cite{FBRS24} and \cite{LM} hold without any  positivity assumption on the Schr\"odinger operators.}	
	
	(a) 
 In view of Corollary \ref{Cor:FBR}, one can see that the summability criteria on $u$ given in \cite{FBRS24} ensures that the assumptions of Theorem \ref{mainZd} are satisfied, and hence $u$ is trivial on $\mathbb{Z}^d$.

(b) {In \cite{LM}, Lyubarskii and Malinnikova proved that if $u$ is $\mathcal{H}$-harmonic on $\Z^d$ and satisfies the decay condition  
$$ {\liminf_{N \rightarrow \infty} \frac{1}{N}\log \max_{|x|_\infty \in \{N,N+1\}}
	|u(x)| < -\|V\|_{\infty}-4d+1 }, $$
then $u=0$. For $\|V\|_{\infty} \leq 1$, this result also follows from Theorem \ref{mainZd}.} 

To see this, assume that the above estimate is satisfied and $\|V\|_{\infty} \leq 1$. We first show that the $\liminf$ condition of Theorem \ref{mainZd} is satisfied.
Indeed, using the inequality  $|x|_{\infty} \ge  |x|/\sqrt{d} $, it follows that for any sequence  of vertices $ (x_{k}) $ 
realizing the $ \liminf $, we have 
\begin{align*}
	|u(x_{k})||x_{k}|^{\frac{d-1}{2}}\mathrm{e}^{|x_k|}\leq |x_{k}|^{\frac{d-1}{2}}\mathrm{e}^{|x_k| - (4d-1)|x_{k}|_{\infty}} \leq |x_{k}|^{\frac{d-1}{2}}{\mathrm{e}^{-2|x_k|}}\to 0
\end{align*}
as $ k\to \infty$. Thus, the liminf condition of Theorem \ref{mainZd} is satisfied. 

{Next, denote $M_N=\max_{|x|_\infty \in \{N,N-1\}}|u(x)|$ for $N \in \N$. We first claim that $M_N \leq {\mathrm{e}}^{- N(4d-1+\|V\|_{\infty})}$ for $N$ sufficiently large.  Otherwise, there exists $N_0 \in \N$ such that $M_{N_0} > {\mathrm{e}}^{-N_0(4d-1+\|V\|_{\infty})}$. On the other hand, it has been shown in the proof of \cite[Lemma 7.2]{FBRS24} that
$$ M_{N+1} \geq  \frac{M_{N}}{(4d-1+\|V\|_{\infty}) } \geq M_{N} {\mathrm{e}}^{-(4d-1+\|V\|_{\infty})} \,.$$
 It follows that  $M_{N_0+1}> {\mathrm{e}}^{-(N_0+1)(4d-1+\|V\|_{\infty})}$ and by induction, 
 $$M_{N_0+n}> {\mathrm{e}}^{-(N_0+n)(4d-1+\|V\|_{\infty})}, \quad   n \in \N.$$ 
 But this contradicts the above liminf condition of \cite{LM}. This proves our claim. }
 
{ Hence, we infer that there exists $C>0$ such that $|u(x)| \leq C {\mathrm{e}}^{-|x|_{\infty}(4d-1+\|V\|_{\infty})} $ on $\mathbb{Z}^d$. This in particular gives $|u(x)| \leq C {\mathrm{e}}^{-4\sqrt{d}|x|} $ on $\mathbb{Z}^d$ for some $C>0$, and consequently, $u \in O(|x|^{(2-d)/2})$. Thus, $u=0$ by Theorem \ref{mainZd}.}
\end{rem}
%%%%%%%%%%%%%%%%%%%%%%%
\subsection{Regular trees}
We consider a tree  with countably infinite vertex set $ X $. We fix an arbitrary vertex $ o $ and denote by $ |x|= d(x,o) $ the {\it combinatorial graph distance} to $ o $ from a vertex $ x $. A tree is said to be {\it {$d$-}regular} if every  vertex has degree $d\in\mathbb{N}$.
Again, we consider the Laplacian $ \Delta $ with standard weights $ b(x,y)=1 $ if $d(x,y)= 1$ and $0$ otherwise as well as  $ m=1 $.

\begin{thm}\label{t:trees} Let $u$ be an $ \mathcal{H} $-harmonic function  of a positive Schr\"odinger operator $ \mathcal{H} =\Delta+V  $ with $ V\leq 1 $ on a $d$-regular tree and $$  u\in O \left( |x|^{\frac{1}{2}}d^{- \frac{|x|}{2}}\right) \quad \mbox{and}\quad \liminf_{x\to\infty}|u(x)|d^{|x|}  =0$$
is trivial, i.e. $ u=0 $.
\end{thm}
\begin{proof} We aim to apply Corollary~\ref{cor:landis1}.
We start by estimating the positive minimal Green function $ G_{\alpha} =(\Delta+\alpha)^{-1}1_{o}$ with $ \alpha=1 $. 
For a vertex $ x $, we consider the forward tree $T_x$ of $x $, that is, the subgraph $ b_{x} $ of $ b $ on the vertex set $ X_{x}=\{y\in X\mid d(o,y)=d(o,x)+d(x,y)\} $. Furthermore, we consider the Dirichlet Laplacian $ \Delta_x ^{(D)}$ with respect to the forward tree $ T_{x} $  which is $\Delta_x^{(D)} = \Delta_{b_{x},1} +1_{x}$. Furthermore, we denote by $$  G^{(x)}_{\alpha} =(\Delta_{x}^{(D)}+\alpha)^{-1}1_{x} $$ the Green function for the forward tree of $ x $.
Then, for  $ x $ and the unique path $ o=x_{0}\sim \ldots\sim x_{n}=x $   connecting $ x $ and $ o $, we have
\begin{align*}
	G_{\alpha}(x)=G_{\alpha}(o)\prod_{j=1}^{n}G^{(x_{j})}_{\alpha}(x_{j})
\end{align*}
and we also have for a vertex $ x $ with number of  neighbors $ d(x) $
\begin{align*}
-\frac{1}{	G_{\alpha}^{(x)}(x)}= -\alpha-d(x)+  \sum_{y\in X_{x}\setminus \{x\},y\sim x} G_{\alpha}^{(y)}(y)
\end{align*} 
see e.g. \cite[Equation~(2.9) and (2.10)]{AizenmanSimsWarzel}, \cite[Proposition~2.1]{Klein} or \cite[Proposition~2.7. and Proposition~2.10.]{KellerDiss}.
For a regular tree with degree $ d=d(x) $ and $ \alpha=1 $, we have that all forward Green functions are the same, i.e., $ 	G_{1}^{(x_{j})}(x_{j})=g $ for some $g>0$. Thus, we have to solve
\begin{align*}
	 d g^{2}-({d+1})g+1=0 
\end{align*}
which gives
\begin{align*}
	g_{\pm}=\frac{d+1}{2d}\pm \sqrt{\frac{(d+1)^{2}}{4d^2}-\frac{1}{d}}=\frac{d+1\pm (d-1)}{2d}\, .
\end{align*}
Since  $G_{1}^{(x_{j})} $ is the smallest positive solution to $ (\Delta^{(D)}_{x_{j}} +1)u=1_{x_{j}}$ and $ g_{-}=1/d<1= g_{+} $, we conclude $ G_{1}^{(x_{j})}(x_{j})=g_{-}= {1}/{d}$. Therefore, $G_{1}(x) = G_{1}(o)d^{-|x|}$. Thus,
\begin{align*}
	\liminf_{x\to \infty}\frac{|u(x)|}{G_{1}(x)} \leq C 	\liminf_{x\to \infty}{|u(x)|}d^{|x|}=0.
\end{align*}

Secondly, we consider the ground state $ v $ associated to a critical Hardy weight for $ \Delta $. In the proof of  \cite[Theorem~2.7]{Berchio}, the authors construct a critical Hardy weight $W= \Delta [(\phi_{1}\phi_{2})^{1/2}] /  (\phi_{1}\phi_{2})^{1/2}$  for $ \Delta $. The functions $ \phi_{1} $, $ \phi_{2} $ are given for $ |x|\geq 1 $ by
\begin{align*}
	\phi_{1}(x)=|x|d^{-|x|/2}\qquad \mbox{and}\qquad \phi_{2}=d^{-|x|/2}
\end{align*}
and $ \phi_{1}(o)=\phi_{2}(o) $ are some positive constant.
Clearly, the function 	$$ 	v(x)=(\phi_{1}(x)\phi_{2}(x))^{1/2} =|x|^{1/2}d^{-|x|/2}  $$ is positive and satisfies $$  (\Delta-W)v=\Delta v -\frac{\Delta v}{v}v=0.  $$
Since $ W $ is proved in \cite[Theorem~2.7]{Berchio} to be a critical Hardy weight, $ v $ is an Agmon ground state of $\Delta-W$. Hence, by assumption $ u\in O(v)  $.
Thus,  $ u=0 $ follows from Corollary~\ref{cor:landis1}.
\end{proof}

\section{On the Fractional Laplacian}\label{sec:Frac}
For a given graph $ b $ over $ (X,m) $, let $ \mathcal{Q}_{0} $ be the energy functional of $\Delta $, and let $ L $ be the positive selfadjoint operator on $ \ell^{2}(X,m) $ associated to the form closure  of  $ \mathcal{Q}_{0}\vert_{C_{c}(X)} $ in $ \ell^{2}(X,m) $ (see Section~\ref{sec:graphs}).

For $ \sigma \in (0,1) $, the discrete fractional Laplacian $ L^{\sigma} $ on $ \ell^{2}(X,m) $ is defined via the spectral theorem. Then $L^\gs$ satisfies 
	\begin{align*}
		L^{\sigma}f(x)=\frac{1}{|\Gamma(-\sigma)|}\int_{0}^{\infty}(I-\mathrm{e}^{-tL})f(x)\frac{\dt}{t^{1+\sigma}} \,,
	\end{align*}	
	where $ \mathrm{e}^{-tL} $ is the heat semigroup of the Laplacian $L$.  Note that the semigroup $ \mathrm{e}^{-tL} $ can be extended to $ \ell^{\infty}(X) $ and the graph is called \emph{stochastically complete} if $ \mathrm{e}^{-tL}1=1 $. In this case it can be observed (see for example \cite{KN23}) that
	\begin{align*}
		L^{\sigma}f(x)&=\frac{1}{|\Gamma(-\sigma)|}\int_{0}^{\infty}\left((\mathrm{e}^{-tL}1)f(x) - (\mathrm{e}^{-tL}f)(x)\right)\frac{\dt}{t^{1+\sigma}}\\
		&= \frac{1}{|\Gamma(-\sigma)|}\int_{0}^{\infty}\sum_{y\neq x}(\mathrm{e}^{-tL}1_y(x))(f(x) -f(y)) \frac{\dt}{t^{1+\sigma}}\\
		&=\frac{1}{|\Gamma(-\sigma)|} \sum_{y\in X}b_{\sigma}(x,y) |(\nabla_{xy} f)	
	\end{align*}
	with \begin{align*}
	b_{\sigma}(x,y)= \int_{0}^{\infty}\mathrm{e}^{-tL}1_y(x)\frac{\dt}{t^{1+\sigma}}
	\end{align*}
	which again gives rise to a graph Laplacian for a  weighted graph $ b_{\sigma} $ which we denote by $  \Delta^{\sigma}  =\Delta_{b_{\sigma},0,m} $. Furthermore, this operator  defined on $ \mathcal{F}_{b_{\sigma}  }$ is an extension of $ L^{\sigma} $.  Since the semigroup $ \mathrm{e}^{-tL} $ maps positive functions to strictly positive functions on connected graphs, $ b_{\sigma} $ will be complete, i.e. every two vertices are adjacent. Hence, as the original graph is infinite,  the graph $ b_{\sigma} $ will be non-locally finite.

In view of the above discussion, we will use the notions $ \mathcal{H}^{\sigma} $-harmonic function, Green function $G_{\alpha}^{\sigma}$ of $\Delta^{\sigma} + \alpha $, (sub)criticality of $\Delta^{\sigma}$, critical Hardy weight for $\Delta^{\sigma}$, etc.,  for the corresponding notions associated with  $\Delta_{b_{\sigma},0,m}$. Therefore, the theory developed in Section~\ref{sec:graphs} and in particular   the Liouville comparison principle, Theorem~\ref{thm:comparison}, applies for  fractional Schr\"odinger operators of the form 
$$ \mathcal{H}^{\sigma}=\Gd^{\sigma}+V \,, \ \ \sigma \in (0,1)\,. $$ 
Consequently, we can study the Landis-type unique continuation results for fractional Schr\"odinger equations. 

The {\it{global unique continuation}} property of the fractional Laplacian on a mesh $(h\mathbb{Z})^d$ was recently investigated in \cite{FRR}. It has been noted that although the global unique continuation property holds for    the fractional Laplacian in the continuum, it fails in the discrete case. Nevertheless, the authors  showed in \cite[Proposition 1.2 \& 1.4]{FRR} that certain global unique continuation property holds for $\ell^{2} $-solutions. Moreover, they have also investigated the global unique continuation for $\mathcal{H}^{\sigma}$.

This motivates us to consider below the Landis-type unique continuation property for $\mathcal{H}^{\sigma}$ on the Euclidean lattice $\mathbb{Z}^d$. {Recall that $\mathbb{Z}^d$ is stochastically complete, that is, ${\mathrm{e}}^{-t\Delta}\,1=1$ (this follows for example by expanding $e^{-t\Delta}=\sum_{k=0}^{\infty}(-t\Delta)^{k}/k!$ for the bounded operator $\Delta$).
\begin{thm}\label{mainZdfrac}  Let $\sigma \in (0,1)$ be such that $0<2\sigma < d$, and  $u$ be an  $ \mathcal{H}^{\sigma} $-harmonic function of a  positive fractional Schr\"odinger operator $ \mathcal{H}^{\sigma} =\Delta^{\sigma}+V $ on $\mathbb{Z}^{d}  $ with $ V\leq 1 $. If $ u $ satisfies  $$  u\in O \left({|x|^{2\sigma-d}}\right) \quad\mbox{ and }\quad \liminf_{x\to\infty} |u(x)| |x|^{2\sigma+d}=0, $$
then $ u=0 $.
\end{thm}
\begin{proof}
	It is known that the Green function $G_0^{\sigma}$ of the fractional Laplacian $\Delta^{\sigma}$ on $\mathbb{Z}^d$ behaves as $G_0^{\sigma} \asymp |x|^{2\sigma-d}$ when $0<2\gs<d$, see \cite[Theorem 8]{DER24}. Thus, due to our assumption, $u \in O(G_0^{\sigma})$.
	
	On the other hand, the Green function $G_1^{\sigma}$ of the fractional Laplacian $\Delta^{\sigma}+1$ on $\mathbb{Z}^d$ behaves as $G_1^{\sigma} \asymp |x|^{-(2\sigma+d)}$ \cite[Theorem 7]{DER24}.
	
	Therefore, the conclusion $ u=0$ follows from  Corollary \ref{cor:landis3}.
\end{proof}
\begin{rem}
	Note that the $\liminf$ condition in the above theorem requires only a polynomial  decay in contrast to the non-fractional case $\gs=1$ (see Theorem \ref{mainZd}).
\end{rem}

Next, we would like to discuss an improvement of the above result by replacing  the a priori bound $ u\in O \left({|x|^{2\sigma-d}}\right) $ in this theorem above by the  weaker bound $u \in O(|x|^{(2\sigma -d)/2})$. We may think of constructing a critical Hardy weight $W_{\sigma}$ for $\Delta^\sigma$ and show that the Agmon ground state $v$ of  $\Delta^\sigma-W_{\sigma}$ satisfies $v \asymp |x|^{(2\sigma -d)/2}$. Having the Green function $G_0^{\sigma}$ and its explicit asymptotic, one may anticipate to use a similar supersolution construction as in \cite[Theorem~1.1.]{KellerPinchoverPogorzelski_optimal} to construct an optimal Hardy weight $W_{\sigma}$ and get the desired asymptotic of $v$. This is not immediate, since the graph $b_{\sigma}$ is non-locally finite and hence $G_0^{\sigma}$ fails to satisfy the {\it{bounded-oscillation}} property, which is required for the optimality result in \cite[Theorem~1.1.]{KellerPinchoverPogorzelski_optimal}. However, in one dimension, due to the recent results in \cite{KN23}, we have an optimal Hardy weight $W_{\sigma}$ of $\Delta^{\sigma}$ on $\mathbb{Z}$ and we know the asymptotic of the associated Agmon ground state. Thus, we have the following result.
\begin{thm}\label{fracd=1}
	Let $\sigma \in (0,1/2)$ and $d=1$. Assume that $u$ is an  $ \mathcal{H}^{\sigma} $-harmonic function of a positive fractional Schr\"odinger operator $ \mathcal{H}^{\sigma} =\Delta^{\sigma}+V $ on $\mathbb{Z}$ with $ V\leq 1 $. If  
	$$  u\in O \left({|x|^{(2\sigma -d)/2}}\right) \quad\mbox{ and }\quad \liminf_{x\to\infty} |u(x)| |x|^{1+2\sigma}=0, $$
	then $ u=0 $.
\end{thm}
\begin{proof}
	Using \cite[Theorem 1 and 5]{KN23}, there exists an optimal Hardy weight $W_{\sigma}$ such that the associated Agmon ground state $v$ of $\Delta^{\sigma} - W_{\sigma}$ on $\mathbb{Z}$ satisfies $v \asymp |x|^{(2\sigma -d)/2}$. Thus, due to our assumption, $u \in O(v)$.
	
	On the other hand, we also have that Green function $G_1^{\sigma}$ of the fractional Laplacian $\Delta^{\sigma}+1$ on $\mathbb{Z}$ behaves as $G_1^{\sigma} \asymp |x|^{-(1+2\sigma)}$ \cite[Theorem 7]{DER24}. Hence, Corollary~\ref{cor:landis1} implies that  $u=0$ on $\mathbb{Z}$.
\end{proof}

%%%%%%%%%%%%%%%%%%%%%%%%%
\appendix
\section{Estimates of norms on $\mathbb{Z}^d$ for the resolvent asymptotics}

In \cite{Michta} one finds the following estimate of the $|\cdot|_a$ norm by the $ \ell^{2} $-norm $|\cdot|$ and the $ \ell^{1} $-norm $ \|\cdot\|_{1} $ on $ \mathbb{Z}^d $.
\begin{lem}[Proposition~1.2. in \cite{Michta}]\label{l:estimate0} For all $ a>0 $, we have	
	\begin{align*}
		|x|\leq |x|_{a}\leq \|x\|_{1}\,.
	\end{align*}
\end{lem}
Moreover, we also get the following $ \ell^{2} $-upper bound for $|x|_{a}$\,.
\begin{lem}\label{l:estimate} For $ a^2< 2$, we have	
	\begin{align*}
	m_{a}|x|_{a}\leq  \sqrt{2a^{2}d} |x|\, .
\end{align*}
 Furthermore, for any element $ x $ of an axis $ \{(0,\ldots,x_{j},\ldots,0) \in \mathbb{Z}^{d}\mid x_{j}\in \mathbb{Z} \} $, $ j=1,\ldots,d $, we have
 \begin{align*}
 	|x|_{a}=|x_{j}|\, .
 \end{align*}
\end{lem}
\begin{proof}
	We start with the second statement and consider $ x=(x_{1},0,\ldots,0) \neq0$. For any such element  on the axis, the equation $   \frac{1}{d}\sum_{i=1}^{d}\sqrt{1+x_{i}^{2}r(x)^{2}}=1+a^{2} $ translates to $ d-1+\sqrt{1+x_{1}^{2}r(x)^{2}}=d(1+a^2) $ and, thus,
	\begin{align*}
		r(x)=\frac{1}{|x_{1}|}\sqrt{(1+da^2)^2 -1} \, . 
		%=\frac{\sqrt{d}a}{x_{1}}\sqrt{2+da^2}.
	\end{align*}
Hence, applying the hyperbolic trigonometric formula, $\cosh^{2}- \sinh^{2}=1 $, yields 
	\begin{align*}
		|x|_a\!&=|x_{1}| \frac{\sinh^{-1}\!\left(\! \sqrt{(1+da^2)^2 -1}\right)}{\cosh^{-1}(1+da^{2})} \!=\!
		|x_{1}| \frac{\sinh^{-1}\!\left(\! \sqrt{\cosh^2\left(\cosh^{-1}(1+da^2)\!\right) -1}
			\right)}{\cosh^{-1}(1+da^{2})}\\
		& =|x_{1}| \, .
	\end{align*}  
To obtain the estimate for arbitrary elements, we    determine the critical points of the function
	\begin{align*}
		F(x,r)=\sum_{i=1}^{d}x_{i}\sinh^{-1}\left(x_{i}r\right)
	\end{align*}
for $ x\in \Z^{d} $, $ r\in[0,\infty) $, under the constraints
\begin{align*}
	f(x,r)= \sum_{i=1}^{d}\sqrt{1+x_{i}^{2}r^{2}}-(1+a^{2})d=0,
\qquad\mbox{and}\qquad
	g(x)=\frac{1}{2}\left(\sum_{i=1}^{d}x_{i}^{2}-1\right)=0,
\end{align*}
since $ m_a|x|_{a}=F(x,r(x)) $, $ f(x,r(x))=0 $ is the defining equation for $ r(x) $ and 
$ |x|=1 $ is equivalent to $ g(x)=0 $.
To apply the methods of Lagrange multipliers, we compute the partial derivatives $\partial_{i}= \partial_{x_{i}} $ and $ \partial_{r} $
\begin{align*}
	\partial_{i}F(x,r)&=\sinh^{-1}\left(x_{i}r\right)+\frac{x_{i}r}{\sqrt{1+x_{i}^{2}r^{2}}}\,, \quad
	&&\partial_{r}F(x,r)= \sum_{i=1}^{d}\frac{x_{i}^{2}}{\sqrt{1+x_{i}^{2}r^{2}}}  \,, \\
		\partial_{i}f(x,r)&=\frac{x_{i}r^{2}}{\sqrt{1+x_{i}^{2}r^{2}}} \,,\quad
		&&	\partial_{r}f(x,r)= \sum_{i=1}^{d}\frac{x_{i}^{2}r}{\sqrt{1+x_{i}^{2}r^{2}}} \,,\\
				\partial_{i}g(x)&=x_{i} \,,\quad &&\partial_{r}g(x)=0.\\
\end{align*}
Hence, letting $ f=g=0 $ and for $ \lambda,\mu\in \R $, we solve
\begin{align*}
\qquad	\partial_{i} F+\lambda \partial_{i}f+\mu \partial_{i}g&=0,\tag{I}\\
\qquad	\partial_{r} F+\lambda \partial_{r}f&=0.\tag{II}
\end{align*}
The second equation (II) reads
\begin{align*}
	 \sum_{i=1}^{d}\frac{x_{i}^{2}}{\sqrt{1+x_{i}^{2}r^{2}}}(1+\lambda r)=0
\end{align*}
so either $ x_{i}=0 $ for all $ i $ or 
\begin{align*}
	\lambda=-\frac{1}{r}.
\end{align*}
Now, $ g(x)=0 $ implies that not all $ x_{i} $ can vanish and therefore equation (I) yields 
\begin{align*}
	\sinh^{-1}\left(x_{i}r\right)+\mu x_{i}=0.
\end{align*}
Hence,  $ x_{i}=0 $ for possibly some $ i=1,\ldots,d $ and all other $ x_{i} $ have to agree (since $ \sinh^{-1} $ and the identity are both monotone increasing functions). We infer from this and $g(x)=0  $ that
\begin{align*}
	|x_{i}|=|x_{j}|=\frac{1}{\sqrt{k}}
\end{align*}
for all $ i,j\in \{l \in\{1,\ldots,d\}\mid {x_{l}}\neq 0\}=:K $ and $ k=\# K $. 
Thus, we conclude that the critical points are of the form
%\begin{align*}
%\frac{1}{\sqrt{k}}	(\sigma_{1}1_{1\in K},\ldots,\sigma_{d} 1_{d\in K} ),
%\end{align*}
\begin{align*}
\frac{1}{\sqrt{k}}	(\sigma_{1}1_{K}(1),\ldots,\sigma_{d} 1_{K}(d) ),
\end{align*}
with $ \sigma_{1},\ldots,\sigma_{d}\in\{\pm 1\} $.
Computing $ r $ for these critical points from $ f(x,r)=0 $, we obtain
\begin{align*}
(1+a^{2})d=	\sum_{i\in K}\sqrt{1+{r^{2}}/{k}} + (d-k)=k\sqrt{1+{r^{2}}/{k}}+(d-k).
\end{align*}
Hence,
\begin{align*}
\left(1	+{a^{2}d}/{k}\right)^{2}=1+{r^{2}}/{k}.
\end{align*}
Therefore,
\begin{align*}
r=	\sqrt{\left(	\left(	1	+{a^2d}/{k}\right)^{2}-1\right)k}\, .
\end{align*}
Plugging the critical points $ (x, r) $ into $ F $, we obtain since $ \sinh^{-1} $ is an odd function
\begin{align*}
		F(x,r)&=\sum_{i\in K}\frac{1}{\sqrt{k}}\sinh^{-1}\left(	\sqrt{\left(	\left(	1	+{a^2d}/{k}\right)^{2}-1\right)} \right)\\
		& ={\sqrt{k}}\sinh^{-1}\left(	\sqrt{\left(	\left(	1	+{a^2d}/{k}\right)^{2}-1\right)} \right) 
\end{align*}
and applying the hyperbolic trigonometric formula $\cosh^{2}- \sinh^{2}=1 $  yields
\begin{align*}
	\ldots&={\sqrt{k}}\sinh^{-1}\left(	\sqrt{\left(	\cosh^{2}\left(\cosh^{-1}\left(	1	+{a^2d}/{k}\right)\right)-1\right)} \right)\\&=
	{\sqrt{k}}\sinh^{-1}\left(	\sinh\left(\cosh^{-1}\left(	1	+{a^2d}/{k}\right)\right)\right)\\&=
	{\sqrt{k}}\cosh^{-1}\left(1	+{a^2d}/{k} \right)  :=G(k).
\end{align*}
By expanding $ \cosh^{-1} $ into a power series for $ |t|<2 $
\begin{align*}
	\cosh^{-1}(1 + t) = \sqrt{2  t} \sum_{n=0}^\infty \frac{2^{-n} (-t)^n {(1/2)^{(n)}}}{{(2 n +1)n!}} \,,
\end{align*}
where the Pochhammer symbol $ {s^{(n)}}  $ is given by {$ s^{(n)}= s(s+1)\ldots(s+n-1) $, (it includes $n$ factors).} We obtain for $ t< 2 $ and $ n\in 2\mathbb{N}+1 $ 
\begin{align*}
	 -\frac{2^{-n} t^n (\tfrac12)^{(n)}}{(2 n +1)n!} +  \frac{2^{-n-1} t^{n+1} (\tfrac12)^{(n+1)}}{(2n+3)(n+1)!}&= -\frac{2^{-n} t^n (\tfrac12)^{(n)}}{(2n+1)n!}\!\left(\!1\!-\!\frac{t(\tfrac12+n) (2n+1)n!}{2(2n+3)(n+1)!}
	  \!\right) \!\leq \!0.	  
\end{align*} 
Thus, since  $ a^2d/k\leq a^{2}< 2 $, we can estimate the series by the first summand and obtain under the constraints $ f=g=0 $
\begin{align*}
	F(x,r)\leq 	G(k) \leq  \sqrt{2a^{2}d}  \,.
\end{align*}
From this, we infer for $ |x| =1 $
\begin{align*}
	m_{a}|x|_{a}= F(x,r(x)) \leq  \sqrt{2a^{2}d}\,, 
\end{align*}
and thus, as $ |\cdot |_{a}$ is a norm, $  m_{a}|x_{a}|\leq  \sqrt{2a^{2}d} |x|$.
This finishes the proof.
\end{proof}
\medskip

{\bf Acknowledgments:} The authors  gratefully acknowledge the financial support of the DFG. U.D. and Y.P.  acknowledge the support of the Israel Science Foundation (grant 637/19) founded by the Israel Academy of Sciences and Humanities. U.D. is also supported in part by a fellowship from the Lady Davis Foundation. 

{The authors are  indebted to the anonymous referee for providing insightful comments that improved the quality of the manuscript.}  
The authors also want to thank Rupert Frank and Dirk Hundertmark for the helpful references to the literature as well as Ira Herbst and Richard Froese for explanations on their earlier results.

{\bf  Data availability statement:}  The authors declare that the data supporting the findings of this study are available within the paper.

{\bf  Conflict of interest statement:} The authors have no conflicts of interest to declare.

\bibliographystyle{alpha}

%\bibliography{literature.bib}

\begin{thebibliography}{FBRS24}

\bibitem[ASW06]{AizenmanSimsWarzel}
M. Aizenman, R. Sims, and S. Warzel.
\newblock Stability of the absolutely continuous spectrum of random
{S}chr\"{o}dinger operators on tree graphs.
\newblock {\em Probab. Theory Related Fields}, 136(3): 363--394, 2006.

\bibitem[ABG19]{ABG} A.~Arapostathis, A.~Biswas, and D.~Ganguly. Certain Liouville properties of
eigenfunctions of elliptic operators. {\em Trans. Amer. Math. Soc.},  371: 4377--4409, 2019.

%\bibitem[BS23]{Balch} K. Le Balch, and D. A. Souza. Quantitative unique continuation for real-valued solutions to second order elliptic equations in the plane. {\em {arXiv}}: 2401.00441, 2024.


\bibitem[BSV21]{Berchio} E.	Berchio, F. Santagati, and M. Vallarino. \newblock Poincar\'e and Hardy inequalities on homogeneous trees. \newblock{ \em Geometric Properties for Parabolic and Elliptic PDE's.}, Cham: Springer International Publishing, 1--22, 2021.

\bibitem[BMP24]{BMP24} S. Biagi, G. Meglioli, and F. Punzo. A Liouville theorem for elliptic equations with a potential on infinite graphs. {\em Calc. Var. Partial Differential Equations} (to appear), 2024.

\bibitem[BP24]{BP24} S. Biagi, and F. Punzo. Phragm\`{e}n-Lindel\"{o}f type theorems for elliptic equations on infinite graphs.  {\em arXiv}: 2406.06505, 15 pp., 2024.


\bibitem[BK05]{BK} J. Bourgain, and C. Kenig. On localization in the continuous Anderson-Bernoulli model in higher dimension. {\em Invent. Math.},  161: 389--426, 2005.

\bibitem[C22]{Cassano} B. Cassano. Sharp exponential decay for solutions
of the stationary perturbed Dirac equation. {\em Commun. Contemp. Math.}, 24(2): 2050078, 2022.
	
	
\bibitem[CR18]{CR18}
\'O. Ciaurri and L.~Roncal.
\newblock {Hardy's inequality for the fractional powers of a discrete
	Laplacian}.
\newblock {\em The Journal of Analysis}, 26: 211--225, 2018.	
	
\bibitem[DKP24]{Das_etal}
U. Das, M. Keller, and Y. Pinchover.
\newblock The space of Hardy-weights for quasilinear operators on discrete graphs.
\newblock {\em arXiv:} 2407.02116, 24 pp., 2024.

\bibitem[DP24]{DP24}
U. Das,  and Y. Pinchover.
\newblock The Landis conjecture via Liouville comparison principle and criticality theory.  {\em arXiv}: 2405.11695, 21 pp., 2024.

\bibitem[DKW17]{Davey1}  B.~Davey, C.~Kenig, and J.N. Wang. The Landis conjecture for variable coefficient second-order elliptic PDEs. {\em Trans. Amer. Math. Soc.},  369: 8209--8237, 2017.

\bibitem[DER24]{DER24} M. Disertori, R. M. Escobar, and C. Rojas-Molina. Decay of the Green's function of the fractional Anderson model and connection to long-range SAW. {\em J. Stat. Phys.}, 191(33): 26 pp., 2024.

\bibitem[DMR24]{DMR24} M.~Disertori, Margherita, R.~Maturana Escobar, and C.~Rojas-Molina. 
Decay of the Green's function of the fractional Anderson model and connection to long-range SAW. 
{\em J.~Stat. Phys.}, 191(3): Paper No. 33, 25 pp. 2024.

\bibitem[DM06]{DM06} J.~Dodziuk, and V.~Mathai. 
Kato's inequality and asymptotic spectral properties for discrete magnetic Laplacians. In: The Ubiquitous Heat Kernel, 69--81. Contemp. Math., 398, Amer. Math. Soc., Providence, RI, 2006.

\bibitem[EKPV10]{EKPV10} L. Escauriaza, C. Kenig, G. Ponce, and L. Vega. The sharp Hardy uncertainty principle for Schr\"odinger evolutions. {\em Duke Math. J.},  155: 163--187, 2010.

\bibitem[EKPV16]{EKPV16} L. Escauriaza, C. Kenig, G. Ponce, and L. Vega. Hardy uncertainty principle, convexity and parabolic evolutions. {\em Comm. Math. Phys.}, { 346}: 667--678, 2016.

\bibitem[FB19]{FB} A. Fern\'andez-Bertolin. A discrete Hardy's uncertainty principle and discrete evolutions. {\em J. Anal. Math.}, 137(2): 507--528, 2019.

\bibitem[FV17]{FV} A. Fern\'andez-Bertolin, and L. Vega. Uniqueness properties for discrete equations and Carleman estimates. {\em J. Funct. Anal.}, { 272}: 4853--4869, 2017.
	

\bibitem[FRR24]{FRR} A. Fern\'{a}ndez-Bertolin, L. Roncal, and A. R\"uland.
	On (global) unique continuation properties of the fractional discrete Laplacian. {\em J. Funct. Anal.},	286(9): 110375, 2024.

\bibitem[FBRS24]{FBRS24} 
A. Fern{\'a}ndez-Bertolin, L. Roncal, and D. Stan.
\newblock Landis-type results for discrete equations. {\em arXiv}: 2401.09066, 40 pp., 2024.


\bibitem[FBRS24b]{FBSR24}
{A. Fern\'andez-Bertolin, L. Roncal, and D. Stan. \newblock{Landis' conjecture: a survey.} {\em arXiv}: 2412.00788, 2024.}

\bibitem[FK23]{Fil} N. D. Filonov, S. T. Krymskii. On the Landis conjecture in a half-cylinder. {\em arXiv}: 2311.14491, 26 pp., 2023.

\bibitem[F24]{F24} {F. Fischer. Quasi-Linear Criticality Theory and Green's Functions on Graphs. \emph{J. Analyse Math.} (to appear), 2024}

\bibitem[FHHH82]{FHH2O2} R. Froese, I. Herbst, Ira, M. Hoffmann-Ostenhof, T. Hoffmann-Ostenhof, 
On the absence of positive eigenvalues for one-body Schr\"{o}dinger operators.
\emph{J. Analyse Math.}, 41: 272--284, 1982.


\bibitem[HKLW12]{HKLW} S. Haeseler, M. Keller, D. Lenz,  R.~K.~Wojciechowski. Laplacians on infinite graphs: Dirichlet and Neumann boundary conditions. {\em J. Spectr. Theory}, 2(4): 397–432, 2012.

\bibitem[JLMP18]{JLMP18} P. Jaming, Y. Lyubarskii, E. Malinnikova, and K.M. Perfekt. Uniqueness for discrete Schr\"odinger evolutions. {\em Rev. Mat. Iberoam.}, 34: 949--966, 2018. 

\bibitem[Kel]{KellerDiss}
M. Keller.
\newblock On the spectral theory of operators on trees.
\newblock Dissertation, Jena 2010,
	https://nbn-resolving.org/urn:nbn:de:gbv:27-20110916-150724-2,
	{\em arXiv}: 1101.2975.


	
\bibitem[KLW21]{KellerBook}
M. Keller, D. Lenz, and R.K. Wojciechowski.
\newblock { Graphs and Discrete Dirichlet Spaces}, {\em
	Grundlehren der mathematischen Wissenschaften [Comprehensive Studies in Mathematics]}. 
\newblock Springer, Cham, 2021.


\bibitem[KN23]{KN23} M.~Keller, and M.~Nietschmann.
Optimal Hardy inequality for fractional Laplacians on the integers. {\em Ann. Henri Poincar\'{e}}, 24(8):  2729--2741. 2023.


\bibitem[PPV24]{PPV} {N. De Ponti, S. Pigola, and G. Veronelli.
\newblock{Unique continuation at infinity: Carleman estimates on general warped cylinders.}
\newblock{\em Int. Math. Res. Not. IMRN}, 2024(16): 11910--11932, 2024.}

\bibitem[KPP18]{KellerPinchoverPogorzelski_optimal}
M. Keller, Y. Pinchover, and F. Pogorzelski.
\newblock Optimal {H}ardy inequalities for {S}chr\"{o}dinger operators on graphs.
\newblock {\em Comm. Math. Phys.}, 358(2): 767--790, 2018.

	
\bibitem[KPP20]{KellerPinchoverPogorzelski}
M. Keller, Y. Pinchover, and F. Pogorzelski.
\newblock Criticality theory for {S}chr\"{o}dinger operators on graphs.
\newblock {\em J. Spectr. Theory}, 10(1): 73--114, 2020.

 \bibitem[KSW15]{KSW} C. E. Kenig, L. Silvestre, and J.N. Wang. On Landis' conjecture in the plane. {\em Comm.  Partial Differential Equations}, { 40}: 766--789, 2015.

\bibitem[Kle98]{Klein}
A. Klein.
\newblock Extended states in the {A}nderson model on the {B}ethe lattice.
\newblock {\em Adv. Math.}, 133(1): 163--184, 1998.


%\bibitem  V. A. Kondrat'ev, E. M. Landis, \newblock {  Qualitative theory of second-order linear partial differential equations} 
%\newblock {\em  Itogi Nauki i Tekhniki [Progress in Science and Technology]} Sovrem. Probl. Mat. Fund. Naprav., 32[Current Problems in Mathematics. Fundamental Directions] Akad. Nauk SSSR, Vsesoyuz. Inst. Nauchn. i Tekhn. Inform., Moscow, 1988, 99–215, 220.

\bibitem[KL88]{Landis2}  V. A. Kondrat’ev and E. M. Landis, Qualitative theory of second order linear partial differential equations.
Partial differential equations, 3 (Russian), 99–215, 220, Itogi Nauki i Tekhniki, Sovrem. Probl. Mat. Fund. Naprav.,
32, Akad. Nauk SSSR, Vsesoyuz. Inst. Nauchn. i Tekhn. Inform., Moscow, 1988.


\bibitem[K22]{KOW}
{P-Z. Kow. On the Landis conjecture for the fractional Schr\"odinger equation. {\em J. Spectr. Theory}, 12(3): 1023--1077, 2022.}

\bibitem[LMNN20]{LMNN} A. Logunov, E. Malinnikova, N. Nadirashvili, and F. Nazarov.  \newblock {The Landis conjecture on exponential decay.} \newblock {\em arXiv}: 2007.07034, 40 pp., 2020.


\bibitem[LM18]{LM} Y. Lyubarskii, and E. Malinnikova. Sharp uniqueness results for discrete evolutions. In:  Non-linear Partial Differential Equations, Mathematical Physics, and Stochastic Analysis, {\em EMS Ser. Congr. Rep., Eur. Math. Society (EMS), Z\"urich}, 423--436, 2018.



\bibitem[M92]{Meshkov}  V. Z. Meshkov. \newblock On the possible rate of decay at infinity of solutions of second order partial differential equations. \newblock{\em Math. USSR Sbornik}, 72: 343--360, 1992.



\bibitem[MS22]{Michta}
E. Michta, and G. Slade.
\newblock Asymptotic behaviour of the lattice {G}reen function.
\newblock {\em ALEA Lat. Am. J. Probab. Math. Stat.}, 19(1): 957--981, 2022.

\bibitem[MY12]{Molchanov}
S.~A. Molchanov, and E.~B. Yarovaya.
\newblock Limit theorems for the {G}reen function of the lattice {L}aplacian
under large deviations for a random walk.
\newblock {\em Izv. Ross. Akad. Nauk Ser. Mat.}, 76(6): 123--152, 2012.

\bibitem[P07]{Pinchover} 
Y.~Pinchover. 
\newblock {A {L}iouville-type theorem for {S}chr\"{o}dinger operators}.
\newblock {\em Comm. Math. Phys.}, 272(1): 75--84, 2007.

\bibitem[R21]{Rossi}L. Rossi. 
\newblock The Landis conjecture with sharp rate of decay.
\newblock{\em	Indiana Univ. Math. J.}, 70: 301--324, 2021.

\bibitem[RW19]{RW19} {A. R\"uland and J-N. Wang. \newblock{On the fractional Landis conjecture.} {\em J. Funct. Anal.,} 277(9): 3236--3270, 2019.}
	
\bibitem[SS21]{Sirakov} B. Sirakov, and P. Souplet.  The V\'azquez maximum principle and the Landis conjecture for elliptic PDE with unbounded coefficients. {\em Adv. Math.}, 387: 107838, 2021.


\bibitem[Woe00]{Woess} W. Woess.
\newblock {\em Random Walks on Infinite Graphs and Groups}.
\newblock Cambridge Tracts in Mathematics. Cambridge University Press, 2000.


\end{thebibliography}

\end{document}